\documentclass[10pt]{amsart}
\usepackage{amsmath,mathtools,amssymb,amsthm}
\usepackage{bbm}
\usepackage{comment,ulem}
\usepackage[thinlines]{easybmat}
\usepackage[vcentermath]{youngtab}
\usepackage{booktabs,longtable}
\usepackage{tikz-cd}
\usepackage{tikz}
\tikzset{
    every picture/.style={baseline=11pt,scale=.5}
}
\usetikzlibrary{calc,shapes,arrows,fit}
\usepackage{pgfplots}
\pgfplotsset{compat=1.7}
\usepackage{graphicx}
\usepackage{xcolor}
\definecolor{airforceblue}{rgb}{0.36, 0.54, 0.66}
\usepackage[pagebackref=true,backref=page,citebordercolor={0.6 0.4 0.8},linkbordercolor={1 0.8 1}]{hyperref} 
\usepackage{cleveref}
\crefname{maintheorem}{Theorem}{Theorems}
\usepackage{enumitem}
\DeclareMathOperator\Ker{Ker}
\let\Im\relax\DeclareMathOperator{\Im}{Im}

\DeclareMathOperator\Sp{Span}

\DeclareMathOperator\irr{Irr}
\DeclareMathOperator\val{val}

\DeclareMathOperator\Hom{Hom}

\DeclareMathOperator\Gr{Gr}
\DeclareMathOperator{\rk}{rank}

\DeclareMathOperator{\pr}{pr}

\DeclareMathOperator{\mat}{Mat}

\newcommand\n{\mathfrak n}

\newcommand{\g}{\mathfrak g}
\newcommand{\p}{\mathfrak p}

\newcommand\ZZ{\mathbb Z}
\newcommand\NN{\mathbb N}
\newcommand\CC{\mathbb C}
\newcommand\TT{\mathbb T}

\newcommand\OO{\mathbb O}

\newcommand\cO{\mathcal O}
\newcommand\cK{\mathcal K}
\newcommand\cN{\mathcal N}

\newcommand\cT{\mathcal T}
\newcommand\bsm{\begin{smallmatrix}}
\newcommand\esm{\end{smallmatrix}}
\newcommand\xt[1][t]{[\hspace{-.12em}[ {#1} ]\hspace{-.12em}]} 
\newcommand\xT[1][t]{(\!( {#1} )\!)} 

\newcommand{\new}[1]{\textbf{#1}}
\newcommand{\flip}[1]{\reflectbox{\rotatebox[origin=c]{180}{#1}}}
\newtheorem{maintheorem}{Theorem} 

\theoremstyle{plain}
\newtheorem{theorem}{Theorem}
\newtheorem*{theorem*}{Theorem}
\newtheorem{conjecture}{Conjecture}
\newtheorem{proposition}{Proposition}
\newtheorem{lemma}{Lemma}
\newtheorem{corollary}{Corollary}

\theoremstyle{remark}

\newtheorem{example}{Example}

\usepackage{fullpage}
\newcommand\ot{\overline{\mathbb O}_\lambda\cap\mathbb T_\mu}

\newcommand{\mvy}{Mirkovi\'c--Vybornov~}

\newcommand{\mvi}{Mirkovi\'c--Vilonen~}
\title{Generalized orbital varieties for Mirkovi\'c--Vybornov slices as affinizations of Mirkovi\'c--Vilonen cycles}
\author{Anne Dranowski}
\email{adranows@math.toronto.edu}
\address{Department of Mathematics, University of Toronto, Room 6135, 40 St. George Street, Toronto, Ontario, Canada M5S 2E4}
\begin{document}
\begin{abstract}
    We show that generalized orbital varieties for \mvy slices can be indexed by semi-standard Young tableaux. We also check that the \mvy isomorphism sends generalized orbital varieties to (dense subsets of) \mvi cycles, such that the (combinatorial) Lusztig datum of a generalized orbital variety, which it inherits from its tableau, is equal to the (geometric) Lusztig datum of its MV cycle. 
\end{abstract}
%
\date{\today}
\maketitle
\tableofcontents
\section{Introduction}
\label{s:intro}
\subsection{Main result}
\label{ss:mainr}
In this paper, we show that generalized orbital varieties for \mvy slices are in bijection with semi-standard Young tableaux, and, via the \mvy isomorphism \cite{mirkovic2007geometric}, can be identified with MV cycles.

Let \(\cT(\lambda)_\mu\) denote the set of semi-standard Young tableaux of shape \(\lambda = (\lambda_1\ge\lambda_2\ge\cdots\ge\lambda_\ell)\) and weight \(\mu = (\mu_1\ge\mu_2\ge\cdots\ge\mu_m)\).\footnote{The weight of a tableau in the alphabet \(\{1,2,\ldots, m\}\) is the \(m\)-tuple of non-negative integers whose \(i\)th entry is the number of times \(i\) appears in the tableau.} Let 
\(N = \sum_1^\ell\lambda_i = \sum_1^m\mu_i\).
We order repeated entries of a tableau from left to right, so that the first occurrence of a given entry is its leftmost.
Then, for \(\tau\in\cT(\lambda)_\mu\) and \((i,k)
\in \{1,2,\ldots,m\}\times\{1,2,\ldots,\mu_i\}
\), we denote by \(\lambda^{(i,k)}_\tau\) (respectively, by \(\mu^{(i,k)}_\tau\)) the shape (respectively, the weight) of the tableau \(\tau^{(i,k)}\) obtained from \(\tau\) by deleting all \(j>i\) and all but the first \(k\) occurrences of \(i\). 
For ease of notation we identify \(\lambda^{(i)}_\tau\equiv \lambda^{(i,\mu_i)}_\tau\), \(\mu^{(i)}_\tau \equiv \mu^{(i,\mu_i)}_\tau\) and \(\tau^{(i)}\equiv\tau^{(i,\mu_i)}\) and when there is no confusion we omit the subscript \(\tau\).

\begin{example}
    \label{tableaunotn}
Let \(\tau = \young(112,23)\). Then \(\tau^{(2)} = \young(112,2)\) has shape \(\lambda^{(2)} = (3,1)\) and weight \(\mu^{(2)} = (2,2)\), while \(\tau^{(2,1)} = \young(11,2)\) has shape \(\lambda^{(2,1)} = (2,1)\) and weight \(\mu^{(2,1)} = (2,1)\).
\end{example}

The array 
\((\lambda^{(1)},\lambda^{(2)},\ldots,\lambda^{(m)})\)
is called the \new{GT-pattern}
of \(\tau\) (see \cite[Section 4]{BERENSTEIN1988453}). 
Note that \(\tau\) can be reconstructed from its GT-pattern. 
Moreover, as we now describe, \(\tau\) defines a matrix variety through its GT-pattern.

Let \(\mat(N)\)
denote the algebra of \(N\times N\) complex matrices.
Let \(\OO_\lambda\) denote the conjugacy class of the Jordan normal form \(J_\lambda\) associated to \(\lambda\), and let \(\TT_\mu\) denote the \mvy slice through the Jordan normal form \(J_\mu\) associated to \(\mu\). Elements of \(\TT_\mu\) take the form \(J_\mu + T\) for \(T\in\mat(N)\) any \(\mu\times\mu\)-block matrix with zeros everywhere except perhaps the first \(\min(\mu_i,\mu_j)\) columns of the last row of the \(\mu_i\times \mu_j\) block, for \(1\le i,j\le m\). 
For example, elements of \(\TT_{(3,2,1)} \) take the form 
\[
\left[ 
    \begin{BMAT}(e){ccc:cc:c}{ccc:cc:c}
    0 & 0 & 0 & 0 & 0 & 0 \\
    0 & 0 & 0 & 0 & 0 & 0 \\
    * & * & * & * & * & * \\
    0 & 0 & 0 & 0 & 0 & 0 \\
    * & * & 0 & * & * & * \\
    * & 0 & 0 & * & 0 & *
    \end{BMAT} 
\right]
\]
with \(*\)s denoting unconstrained entries. Note, \(\dim \TT_\mu = \sum_{i=1}^m(2i-1)\mu_i\).

Let \(\beta_\mu = (e^1_1,\ldots,e^{\mu_1}_1,\ldots,e^1_m,\ldots,e^{\mu_m}_m)\) be a \(\mu\)-enumeration of the standard basis of \(\CC^N\), and let \(V^{(i,k)}\) denote the span of the first \(\mu_1 + \cdots + \mu_{i-1} + k\) vectors of \(\beta_\mu\) for \((i,k)\in \{1,2,\ldots,m\}\times\{1,2,\ldots,\mu_i\}\).
We'll also identify \(V^{(i)}\equiv V^{(i,\mu_i)}\) for all \(i\).

Let \(\n\subset\mat(N) \) denote the subalgebra of upper-triangular matrices and consider 
\begin{equation}\label{eq:blockyXt}
X_\tau = \left\{
A \in \TT_\mu\cap\n : 
A\big|_{V^{(i)}} \in \OO_{\lambda^{(i)}}
\text{ for } 1\le i\le m
\right\}\,.
\end{equation}
Here, we are identifying \(A\big|_{V^{(i)}} \) with the top left \(N_i\times N_i\) submatrix of \(A\) for \(N_i =\sum_{j=1}^i \mu_j =\sum_{j=1}^i \mu^{(i)}_j = \sum_{j=1}^i \lambda^{(i)}_j\,. \)
\begin{example}
Let \(\tau = \young(112,23)\) as before. Then
\[
X_\tau = \left\{
    A = 
\left(\begin{BMAT}(e){ccccc}{ccccc}
  0 & 1 & 0 & 0 & 0 \\
  0 & 0 & a & b & c \\
  0 & 0 & 0 & 1 & 0 \\
  0 & 0 & 0 & 0 & d \\
  0 & 0 & 0 & 0 & 0
    \addpath{(2,5,:)ddrrddr}
    \addpath{(4,5,:)ddr} 
\end{BMAT}\right) : a,d = 0\text{ and } b,c\ne 0
\right\}\,.
\]
\end{example}
\begin{maintheorem}
\label{mainanne1}
\(X_\tau\) has one component \(X_\tau^{d}\) of maximum dimension \(d\) which can be computed from \(\tau\), or (independent of \(\tau\)) from \(\lambda\) and \(\mu\).
Moreover, the closure \(Z_\tau = \overline{X_\tau^{d}}\) is an irreducible component of \(\ot\cap\n\), and 
conversely, every irreducible component of \(\ot\cap\n\) is of this form.
\end{maintheorem}

The latter half of this claim is due to \cite{zinn2015quiver} where it is stated without proof.
We call \(Z_\tau\) a \new{generalized orbital variety for the \mvy slice \(\TT_\mu\)}, for when \(\tau\) is a standard Young tableau, so \(\mu = (1,\ldots,1)\), \(\TT_\mu = \mat(N)\) and the decomposition 
\(\overline\OO_\lambda\cap\n = \cup_{\sigma \in S(\lambda)_{\mu}} Z_\sigma\)
recovers the ordinary orbital varieties of \cite{joseph1984variety} by \cite{spaltenstein1976fixed}. 

Now, \(\lambda\) and \(\mu\) can also be viewed as coweights of \(G = GL(m,\CC)\)
parametrizing MV cycles via their images \(L_\lambda\) and \(L_\mu\) in the affine Grassmannian \(\Gr = G(\cK)/G(\cO)\) of \(G\). Here \(\cO = \CC\xt\) and \(\cK = \CC\xT\).

Let \(T\subset G\) be a maximal torus and consider the homomorphism which identifies \(z_i\in X_\bullet(T) = \Hom(\CC^\times, T)\) and \(e_i\in\ZZ^m\) such that \(z_i(t) = t^{e_i} \) is the diagonal matrix with \((k,k)\) entry equal to \(t\) if \(k=i\) and 1 if \(k\ne i\). Thus \(\nu\in\ZZ^m\) defines \(t^\nu\in G(\cK)\) which in turn defines \(L_\nu = t^\nu G(\cO)\in\Gr\).

Let \(\Gr^\lambda\) denote the \(G(\cO)\) orbit of \(L_\lambda\) and let 
\(S^\mu_-\)
denote the 
\(U_-(\cK)\)
orbit of \(L_\mu\). Here, 
\(U_-\subset G\)
denotes the subgroup of invertible 
lower-triangular matrices.
\new{MV cycles of coweight \((\lambda,\mu)\)} are defined as the irreducible components of 
\(\overline{\Gr^\lambda\cap S^\mu_-}\).
By \cite{mirkovic2007geometric} they give a basis of the \(\mu\)-weight space of the highest weight \(\lambda\) irreducible representation \(L(\lambda)_\mu\) of (the Langlands dual group of \( G\) --- in this case) \(G\). 

Let 
\(\Phi^+ \)
denote the set of positive coroots 
\(\{\alpha_i+\cdots+\alpha_j : 1\le i < j \le m\}\)
for 
\(\alpha_i = z_i - z_{i-1}\) in \(\ZZ^m\).
By \cite[Theorem 4.2]{kamnitzer2010mirkovic} the MV cycles are parametrized by their \new{\(\bf i\)-Lusztig data}, which are \(\Phi^+\)-tuples of non-negative integers ordered by a choice of reduced word \(\bf i\) for the longest element \(w_0\) in the Weyl group of \(G\), and computed intrinsically in \(\Gr\).

Elements of \(\cT(\lambda)_\mu\) also acquire \(\bf i\)-Lusztig data in \(\NN^{\Phi^+}\) from their GT-patterns (see \Cref{s:equat}, \Cref{comld}).

Let us fix the parametrization \({\bf i} = (12\dots m\dots 12 1)\) inducing the order
\[z_1 - z_2 < \cdots < z_1 - z_m < z_2 - z_3 < \cdots < z_2 - z_m < \cdots < z_{m-1} - z_m\,,\]
and henceforth omit \(\bf i\) from the notation.
\begin{maintheorem}
\label{mainanne2}
The \mvy isomorphism, restricts to an isomorphism \(\psi\) of \(\ot\cap\n\) and \(\overline{\Gr^\lambda}\cap S^\mu_- \) such that \(\overline{\psi(Z_\tau)}\) is dense in an MV cycle with Lusztig datum equal to the Lusztig datum of \(\tau\).
\end{maintheorem}
In particular, the \mvy isomorphism induces a Lusztig data preserving bijection between MV cycles of coweight \((\lambda,\mu)\) and semi-standard Young tableaux of shape \(\lambda\) and weight \(\mu\). 
\subsection{Applications and relation to other work}
\label{ss:appli}
\subsubsection{Measures of MV cycles}
\label{sss:measu}
By \cite{mirkovic2007geometric}, MV cycles yield a basis in representations of \(G\). In \cite{bkk2019}, the authors show that, combinatorially, this basis is the same as Lusztig's dual semi-canonical basis. In the appendix to \cite{bkk2019}, the appendix authors show that, geometrically, these bases are different. Our comparison relies on \Cref{mainanne1,mainanne2} together with results of \cite{bkk2019} on what geometric equality would entail. In particular, from \Cref{eq:blockyXt} we can determine the ideal \(I_\tau\) of \(X_\tau\). In turn, by normalization we can obtain from \(I_\tau\) the ideal of \(\overline{\psi(Z_\tau)}\). Thus the title of this paper.
\subsubsection{Big Springer fibres}
\label{sss:bigsp}
Set \(GL(N) \equiv GL(N,\CC)\). Given a partition \(\nu\vdash N\), let \(P_\nu\subset GL(N)\) be the corresponding parabolic subgroup and denote by \(\p_\nu\) its Lie algebra. We'll view elements of the partial flag variety \(X_\nu := GL(N)/P_\nu\) interchangeably as parabolic subalgebras of \( \mat(N)\) which are conjugate to \(\p_\nu \) and as flags \(0 = V_0\subset V_1 \subset \cdots \subset V_{\nu_1} = \CC^N \) such that \(\dim V_i/V_{i-1} = (\nu^T)_i\) for \(i = 1,\ldots,\nu_1\). Here \(\nu^T\) denotes the conjugate partition of \(\nu\).

Shimomura, in \cite{shimomura1980theorem}, establishes a bijection between components of \new{big Springer fibres} \((X_\mu)^u\), for fixed \(u-1\in\OO_\lambda\), and \(\cT(\lambda)_\mu\), generalizing Spaltenstein's decomposition in \cite{spaltenstein1976fixed} in case \(\mu = (1,\ldots,1)\), and implying that big Springer fibres also have the same number of top-dimensional components as \(\ot\cap\n\). We conjecture that the coincidence is evidence of a correspondence implying a bijection between the top-dimensional irreducible components of \(\OO_\lambda\cap\p_\mu\) and \(\ot\cap \n\).

Let \(\cN\) denote the nilpotent cone in \(\mat(N)\).
Let 
\( \widetilde{\g}_\mu = \{ (A,V_\bullet)\in \cN \times X_\mu : AV_i\subset V_i \text{ for } i = 1,\ldots, (\mu^T)_1\}\).
Equivalently, 
\(\widetilde{\g}_\mu = \{(A,\mathfrak p)\in \cN \times X_\mu : A\in \mathfrak p\}\). 
Let 
\(A = u-1 \in \OO_\lambda\) and consider the restriction of 
\(\pr_1 : \widetilde{\g}_\mu\to\cN\) defined by \(\pr_1(A,\p) = A\) to 
\(\widetilde{\g}_\mu^\lambda = \pr_1^{-1}(\OO_\lambda)\). 
We conjecture that the (resulting) diagram
\[
\begin{tikzcd}
& \OO_\lambda \cap \p_\mu \ar[r,mapsto] \ar[d,hook] & \{\p_\mu\} \ar[d,hook] \\ 
(X_\mu)^u \ar[d] \ar[r,hook] & \widetilde{\g}_\mu^\lambda \ar[r]\ar[d] & X_\mu \\
\{A\} \ar[r,hook] & \OO_\lambda
\end{tikzcd}    
\]
has an orbit-fibre duality (generalizing the bijections established in \cite[\S 6.5]{chriss2009representation} when \(\mu = (1,\ldots,1)\) and \(\OO_\lambda\cap\p_\mu = \OO_\lambda\cap\n\)) such that the maps
\(\OO_\lambda\cap\p_\mu \to \tilde \g^{\lambda}_\mu \leftarrow (X_\mu)^u\) 
give bijections on top dimensional irreducible components.
\subsubsection{Symplectic duality of small Springer fibres}
\label{sss:sympl}
By \cite[Theorem 5.37]{webster2017generalized}, the restriction of the parabolic analogue of the Grothendieck--Springer resolution \(\pi: T^\ast X_\mu\to \overline\OO_{\mu^t}\) 
to \(\mathfrak{X}_\mu^{\lambda^t} = \pi^{-1}(\overline\OO_{\mu^t}\cap\TT_{\lambda^t}) = \{(A, V_\bullet) \in \overline\OO_{\mu^t}\cap\TT_{\lambda^t} \times X_{\mu} : AV_i\subset V_{i-1}\text{ for all } i = 1,\ldots,(\mu^T)_1\} \) is symplectic dual to 
\(
    \pi^! : \mathfrak{X}_{\lambda^t}^{\mu} 
    \to \ot
\).
A hard consequence of this is that
\(H^{\text{top}}(\pi^{-1}(J_{\lambda^T})) = H_{\text{top}}(\ot\cap\n)\) where 
\(J_{\lambda^T}\in \overline\OO_{\mu^T}\cap\TT_{\lambda^T}\) 
denotes the Jordan normal form associated to \(\lambda^T\).
Haines, in \cite{haines2006equidimensionality}, establishes a bijection between components of \(\pi^{-1}(J_{\lambda^T})\) and \(\cT(\lambda)_\mu\). 
Thus, symplectic duality indirectly predicts that components of \(\ot\cap\n\) are in bijection with \(\cT(\lambda)_\mu\) too. 

\subsection{Acknowledgements}
\label{ss:ackno}
I would like to thank my advisor Joel Kamnitzer. His encouragement and suggestions have been valuable throughout this project.

\section{Generalized orbital varieties for \mvy slices}
\label{s:gener}

We begin by giving a more tractable description of the sets defined by \Cref{eq:blockyXt}.
\subsection{A boxy description of \texorpdfstring{\(X_\tau\)}{Xt}}
\label{ss:aboxy}
\begin{lemma}
    \label{impliesboxyalg}
Let \(B\) be an \((N-1)\times(N-1)\) matrix of the form
\[
\begin{bmatrix}
    C & v \\
    0 & 0 
\end{bmatrix}
\]
for some \((N-2)\times(N-2)\) matrix \(C\) and column vector \(v\). 
Let \(A \) be an \(N\times N\) matrix of the form 
\[
\begin{bmatrix}
    C & v & w \\
    0 & 0 & 1 \\
    0 & 0 & 0 
\end{bmatrix}    
\]
for some column vector \(w\).
Let \(p\ge 2\). If \(\rk C^p < \rk B^p  \), then \(\rk B^p < \rk A^p\). 
\end{lemma}

\begin{proof}[Proof of \Cref{impliesboxyalg}]
Let 
\[
    B = 
\begin{bmatrix}
    C & v \\
    0 & 0 
\end{bmatrix}
\]
and let
\[
A = \begin{bmatrix}
    C & v & w \\
    0 & 0 & 1 \\
    0 & 0 & 0 
\end{bmatrix} 
= 
\left[\begin{BMAT}(@){c:c}{c:c}
    B & \begin{BMAT}{c}{cc}
        w \\ 1
    \end{BMAT} \\
    \begin{BMAT}{cc}{c}
        0 & 0
    \end{BMAT} & 0 
\end{BMAT}\right]\,.
\]
Suppose \(\rk B^p > \rk C^p \) for \(p\ge 0\). Clearly \(\rk A^p > \rk B^p\) for \(p = 0,1\) independent of the assumption. Suppose \(p\ge 2\). 
Since 
\[
B^p = \begin{bmatrix}
    C^p & C^{p-1} v \\
        0 & 0 
\end{bmatrix}
\]
this means \(C^{p-1} v \not\in\Im C^p \). So \(C^{p-2} v\not\in \Im C^{p-1}\) and \(C^{p-2} v + C^{p-1} w \not\in \Im C^{p-1}\).
Since 
\[
A^p = \begin{bmatrix}
C^p & C^{p-1} v & C^{p-1} w + C^{p-2} v\\
    0   & 0         & 0 \\
    0   & 0         & 0 
\end{bmatrix}    
\] 
it follows that \(\rk A^p > \rk B^p \) as desired.
\end{proof}

Now fix \(A\in X_\tau\) with \(\tau\in\cT(\lambda)_\mu\) as above. Recall that \(V^{(i,k)}\) denotes the span of the first \(\mu_1+\cdots + \mu_{i-1} + k\) vectors of \(\beta_\mu = (e^1_1,\ldots,e^{\mu_1}_1,\ldots,e^1_m,\ldots,e^{\mu_m}_m)\).
\begin{lemma}
    \label{boxyalg} \(A\big|_{V^{(m,\mu_m-1)}}\in\OO_{\lambda^{(m,\mu_m - 1)}}\).
\end{lemma}
\begin{proof}
Let \(b = A\big|_{V^{(m-1)}}\) and \(B = A\big|_{V^{(m,\mu_m-1)}}\).
Assume \(\mu_m > 1\) or else \(b = B\) and there is nothing to show. Let \(C = A\big|_{V^{(m,\mu_m-2)}}\).

By definition of \(X_\tau\), \(A\in \OO_\lambda\) and \(b\in\OO_{\lambda^{(m-1)}}\). 
Let \(\lambda(B)\) denote the the Jordan type of \(B\) and \(\lambda(C)\) the Jordan type of \(C\).
Since \(\dim V/V^{(m-1)} = \mu_m \) is exactly the number of boxes by which \(\lambda\) and \(\lambda^{(m-1)}\) differ, \(\lambda(B)\) must contain one less box than \(\lambda\), and \(\lambda(C)\) must contain one less box than \(\lambda(B)\).
Let \(c(A)\) denote the column coordinate of the box by which \(\lambda\) and \(\lambda(B)\) differ, and let \(c(B) \) denote the column coordinate of the box by which \(\lambda(B)\) and \(\lambda(C)\) differ.
Then
\[
\rk B^p - \rk C^p = \begin{cases}
    1 & p < c(B) \\
    0 & p \ge c(B)
\end{cases}\,,
\]
so we can apply \Cref{impliesboxyalg} to our choice of \((A,B,C)\) to conclude that \(\rk A^p > \rk B^p \) for \(p < c(B) \). At the same time,
\[
\rk A^p - \rk B^p = \begin{cases}
    1 & p < c(A) \\
    0 & p \ge c(A)
\end{cases}
\]
implies that \(c(A) > c(B) \). We conclude that \(B \in \OO_{\lambda^{(m,\mu_m - 1)}}\) as desired.
\end{proof}

The blocky rank conditions defining \(X_\tau\) in \Cref{eq:blockyXt} can thus be refined to boxy rank conditions.
\begin{proposition}
\label{boxyXt}
\begin{equation}\label{eq:boxyXt} 
    X_\tau = \left\{
        A\in\TT_\mu\cap\n : A\big|_{V^{(i,k)}}\in\OO_{\lambda^{(i,k)}} \text{ for } 1\le k\le \mu_i \text{ and } 1\le i\le m
    \right\}\,.
\end{equation}
\end{proposition}
\begin{proof}
The non-obvious direction of containment is an immediate consequence of \Cref{boxyalg}.    
\end{proof}
\subsection{Irreducibility of \texorpdfstring{\(X_\tau\)}{Xt}}
\label{ss:irred}
We now prove that our matrix varieties are irreducible in top dimension. For \(1\le i\le m\), let \(\rho^{(i)} = iz_1 + (i-1)z_2 + \cdots + z_i\) in \(\ZZ^i\) be the familiar half sum of positive coroots for \(G_i = GL(i,\CC)\), set \(\rho \equiv \rho^{(m)}\) and let \(\langle\phantom{0},\phantom{0}\rangle\) denote the dot product.
\begin{proposition}
\label{Xtirr}
    \(X_\tau\) has one irreducible component \(X^d_\tau\) of maximum dimension
    \(d = \langle \lambda - \mu, \rho\rangle\).
\end{proposition}
Set \(\tau-\young(m) \equiv \tau^{(m,\mu_m-1)}\), and let \(r\) equal to the row coordinate of the last \(m\), aka the row coordinate of the box by which \(\tau\) and \(\tau - \young(m)\) differ.
\begin{lemma}
\label{irrfib}
The map
\[
X_\tau\to X_{\tau - \young(m)}:A \mapsto A\big|_{V^{(m,\mu_m - 1)}}
\]
has irreducible fibres of dimension \(m - r\).
\end{lemma}
\begin{proof}
Let \(B \in X_{\tau-\young(m)}\) and let \(F_B\) denote the fibre over \(B\).
We'll show that 
\[
    F_B \cong {(B^{\lambda_r - 1})}^{-1}\Im B^{\lambda_r}\cap L \setminus 
    {(B^{\lambda_r - 2})}^{-1}\Im B^{\lambda_r-1}\cap L
\]
for \(L=\Sp_\CC(e^{\mu_1}_1,\ldots,e^{\mu_{m-1}}_{m-1})\).
The dimension count will then follow by \Cref{altfibdes} below. 

Assume \(\mu_m > 1\) and let \(A\in F_B\) take the form 
\[
A = \left[\begin{BMAT}(@){c:c}{c:c}
    B & v + e^{\mu_m - 1}_m \\
    0 & 0
\end{BMAT}\right] =  
\left[\begin{BMAT}(@){c:c}{c:c}
    B & \begin{BMAT}{c}{cc}
        v \\ 1
    \end{BMAT} \\
    \begin{BMAT}{cc}{c}
        0 & 0
    \end{BMAT} & 0 
\end{BMAT}\right]
\]
with 
\(v\in L \). 
Let \(u\in\Ker A^{\lambda_r} \setminus \Ker A^{\lambda_r-1}\) and suppose without loss of generality \(u = e^{\mu_m}_m + w\) for some \(w\in V^{(m,\mu_m-1)}\).
Then 
\[
\begin{aligned}
    0 &= A^{\lambda_r} (u) \\
      &= A^{\lambda_r} ( e^{\mu_m}_m + w ) \\
      &= B^{\lambda_r-1} (v + e^{\mu_m - 1}_m) + B^{\lambda_r} (w) \,.
\end{aligned} 
\]
That is \(v + e^{\mu_m - 1}_m \in {(B^{\lambda_r-1})}^{-1}\Im B^{\lambda_r}\) and \(A\in F_B\) is uniquely specified by an element in
\[
    \left({(B^{\lambda_r-1})}^{-1}\Im B^{\lambda_r} + e^{\mu_m - 1}_m\right)\cap L 
\]
which is isomorphic to \({(B^{\lambda_r-1})}^{-1}\Im B^{\lambda_r} \cap L\) since \(V^{(m,\mu_m-1)} = {(B^{\lambda_r-1})}^{-1}\Im B^{\lambda_r} + L\).

In turn, the isomorphism
\[
    \left({(B^{\lambda_r-1})}^{-1}\Im B^{\lambda_r} \setminus {(B^{\lambda_r-2})}^{-1}\Im B^{\lambda_r-1}\right) \cap L \to F_B : w\mapsto 
        \left[\begin{BMAT}(@){c:c}{c:c}
            B & e^{\mu_m}_m+w \\
            0 & 0
        \end{BMAT}\right]\,,
\]
of \(F_B\) and the locally closed set on the lefthand side, where note \({(B^{\lambda_r-2})}^{-1}\Im B^{\lambda_r-1}\) is excluded, since 
\(A^{\lambda_r-1}(u)\ne 0\), proves that \(F_B\) irreducible. By \Cref{altfibdes} below, \(F_B\) has dimension \(m - r\).
\end{proof}
\begin{lemma}
\label{altfibdes}
Let \(B\in X_{\tau - \young(m)}\). Then
\begin{subequations}
    \begin{align}
        \label{dim1}
        \dim {(B^{\lambda_r - 1})}^{-1}\Im B^{\lambda_r} &= N - r\,,\text{ and} \\
        \label{dim2}
        \dim {(B^{\lambda_r - 1})}^{-1}\Im B^{\lambda_r}\cap L &= m - r\,.
    \end{align}
\end{subequations}
\end{lemma}
\begin{proof}
    By \Cref{boxyalg}, 
    \(B\)
    has Jordan type \(\lambda^{(m,\mu_m - 1)}\) 
    which differs from \(\lambda\) by a single box in position 
    \((r,\lambda_r)\).
    Let \(J = (f_1^1,\ldots,f^{\lambda_1}_1,\ldots,f_r^1,\ldots,f_r^{\lambda_r - 1},\ldots,f_\ell^1,\ldots,f_\ell^{\lambda_\ell})\) be a Jordan basis for \(B\). Then, with respect to \(J\),
\[
\begin{aligned}
    \Im B^{\lambda_r - 1}  
    &= \Sp_\CC (f_c^1 ,\ldots, f_c^{\lambda_c - \lambda_r + 1} : 1\le c\le\ell) \\
    &= 
        \Sp_\CC(f_c^1 ,\ldots, f_c^{\lambda_c - \lambda_r}  : 1\le c\le\ell)
    + \Sp_\CC (f_c^{\lambda_c - \lambda_r + 1}  : 1\le c\le\ell) \\
    &= \Im B^{\lambda_r} + \Sp_\CC (f_c^{\lambda_c - \lambda_r + 1}  : 1\le c\le\ell)\,,
\end{aligned}
\]
where we understand \(f_c^p \equiv 0\) for \(p\le 0\). In particular, \(f_c^{\lambda_c - \lambda_r + 1}\) is equal to \(B^{\lambda_r - 1}(f_c^{\lambda_c})\) and is nonzero for \(c > r\). 
Thus \(\dim {(B^{\lambda_r - 1})}^{-1}\Im B^{\lambda_r} = N - r\).

Let \(A\in F_B\). We claim that
\[
V^{(m,\mu_m-1)} = (A\big|_{V^{(m,\mu_m - 1)}}^{\lambda_r - 1})\Im A\big|_{V^{(m,\mu_m - 1)}}^{\lambda_r} + L \,.
\]
Let \(e_a^b\in\beta_\mu \). If \(b\le\mu_a - 1 \) then \(e_a^b = A\big|_{V^{(m,\mu_m - 1)}} (e_a^{b+1}) - v\) for some \(v\in L\) . 
Since \(A\big|_{V^{(m,\mu_m - 1)}}(e_a^{b+1})\) is clearly in \((A\big|_{V^{(m,\mu_m - 1)}}^{c-1})^{-1}\Im A\big|_{V^{(m,\mu_m - 1)}}^{c}\) for any \(c\) it follows that  \(e_a^b \in (A\big|_{V^{(m,\mu_m - 1)}}^{c-1})^{-1}\Im A\big|_{V^{(m,\mu_m - 1)}}^{c} + L \) for any \(1\le b\le\mu_a\) and \(1\le a\le m\) except of course for \((a,b) = (m,\mu_m)\).

We can therefore apply the elementary fact that the codimension of \(V'\) in \(V' + V''\) is equal to the codimension of \(V'\cap V''\) in \(V''\) for any two vector spaces \(V'\) and \(V''\) to \(V' = (A\big|_{V^{(m,\mu_m - 1)}}^{\lambda_r-1})^{-1}\Im A\big|_{V^{(m,\mu_m - 1)}}^{\lambda_r}\) and \(V'' = L\). 

Together with \Cref{dim1} this gives \Cref{dim2}:
\[
\begin{aligned}
\dim V'\cap V'' &= \dim V' + \dim V'' - \dim (V' + V'') \\
                &= (N-r) + (m-1) - (N-1)\\
                &= m - r\,.
\end{aligned}
\]%
\end{proof}
We would like to use \Cref{irrfib} to establish \Cref{Xtirr}. To do so we will need the following proposition.
\begin{proposition}
\label{myirrcl} 
Let \(f:X\to Y\) be surjective, with irreducible fibres of dimension \(d\). Assume \(Y\) has a component of dimension \(m\) and all other components of \(Y\) have smaller dimension. Then \(X\) has unique component of dimension \(m + d\) and all other components of \(X\) have smaller dimension.
\end{proposition}
To prove \Cref{myirrcl}  will need \cite[I, \S 8, Theorem 2]{mumford1988red} and \cite[Lemma 005K]{stacks-project} which we now recall.
\begin{theorem}
\label{mumf2}
\cite[I, \S 8, Theorem 2]{mumford1988red}
Let \(f:X\to Y\) be a dominating morphism of varieties and let \(r = \dim X - \dim Y\). Then there exists a nonempty open \(U\subset Y \) such that:
\begin{enumerate}
    \item \(U\subset f(X)\)
    \item for all irreducible closed subsets \(W\subset Y\) such that \(W\cap U \ne \varnothing\), and for all components \(Z\) of \(f^{-1}(W)\) such that \(Z\cap f^{-1}(U)\ne\varnothing \), \(\dim Z = \dim W + r\). 
\end{enumerate}
\end{theorem}
\begin{lemma}
\label{lemma-generic-point-in-constructible}
\cite[Lemma 005K]{stacks-project}
Let $X$ be a topological space. Suppose that
$Z \subset X$ is irreducible. Let $E \subset X$
be a finite union of locally closed subsets (e.g.\ $E$
is constructible). The following are equivalent
\begin{enumerate}
\item The intersection $E \cap Z$ contains an open
dense subset of $Z$.
\item The intersection $E \cap Z$ is dense in $Z$.
\end{enumerate}
\end{lemma}
\begin{proof}[Proof of \Cref{myirrcl}]
Let \(X = \cup_{\irr X} C\) be a (finite) decomposition of \(X\). Consider the restriction \(f\big|_C : C \to \overline{f(C)}\) of \(f\) to an arbitrary component. It is a dominant morphism of varieties, with irreducible fibres of dimension \(d\).

We apply \Cref{mumf2}. Let \(U\subset \overline{f(C)}\) be such that for all irreducible closed subsets \(W\subset \overline{f(C)} \) such that \(W\cap U\ne \varnothing\), and for all components \(Z\) of \(f^{-1}(W)\) such that \(Z\cap f^{-1}(U) \ne \varnothing\), \(\dim Z = \dim W + \dim C - \dim \overline{f(C)}\). 
Then, taking \(W = \{y\} \subset U\) for some \(y\in U\subset \overline{f(C)}\), we get that \(\dim f^{-1}(y) = \dim C - \dim \overline{f(C)}\). Since all fibres have dimension \(d\), the difference \(\dim C - \dim \overline{f(C)}\) is constant and equal to \(d\), independent of the component we're in. 

Since \(f\) is surjective, it is in particular dominant, so we have that 
\[
Y = f(X) = f(\cup_{\irr X} C) = \cup_{\irr X} f(C) = \overline{\cup_{\irr X} f(C)}= \cup_{\irr X} \overline{f(C)}\,.
\]
Let \(C = C_0\) be such that \(\dim \overline{f(C_0)} = m\).
Then \(\dim C_0 = d + \dim \overline{f(C_0)} = d + m\).

Let \(f_i = f\big|_{C_i}\) and let \(U_i \subset \overline{f_i(C_i)} \) be the open sets supplied by \Cref{mumf2} or \Cref{lemma-generic-point-in-constructible} for the constructible sets \(E_i = f(C_i)\). Take \( U = U_0\cap U_1\) and let \(y\in U\). Since \(V_i = f_i^{-1}(U) \) contains \(f_i^{-1}(y) = f^{-1}(y)\cap C_i = f^{-1}(y)\) the set \(V = V_0\cap V_1\) is nonempty. That's a nonempty open set contained in \(C_0\cap C_1\). Conclude \(C_0 = C_1\). Note \(V_i =  f^{-1}(U) \cap C_i\).
\end{proof}
We are now ready to prove \Cref{Xtirr}. 
\begin{proof}[Proof of \Cref{Xtirr}]
Consider the restriction map
\[
X_\tau \to X_{\tau^{(m-1)}} \,.
\]
By induction on \(m\), we can assume that \(X_{\tau^{(m-1)}}\) has one irreducible component of dimension 
\(d \),
and apply \Cref{myirrcl} in conjunction with \Cref{irrfib} to conclude that \(X_\tau\) has one irreducible component of dimension \(d+\sum_1^{\mu_m}(m-r_{m,k})\) for \(r_{m,k}\) equal to the row coordinate of the \(k\)th \(m\) in \(\tau\). Note \(X_{\tau^{(1)}} = \{J_{\mu_1}\}\).

We now check that
\(\sum_1^{\mu_m} (m-r_{m,k}) = \langle \lambda-\mu,\rho\rangle - \langle \lambda^{(m-1)}-\mu^{(m-1)},\rho^{(m-1)}\rangle\). We start by expanding the difference on the righthand side.
\[
\begin{aligned}
    \langle \lambda - \mu,\rho \rangle 
    &- \langle \lambda^{(m-1)} - \mu^{(m-1)} ,\rho^{(m-1)}\rangle \\
    &= \lambda_1 - \mu_1 + (m-1) (\lambda_1 - \lambda^{(m-1)}_1) \\  
    &\quad+ \lambda_2 - \mu_2 + (m-2)(\lambda_2 - \lambda^{(m-1)}_2) \\
    &\quad\,\,\,\vdots \\
    &\quad+ \lambda_{m-1} - \mu_{m-1} + (\lambda_{m-1} - \lambda^{(m-1)}_{m-1}) \\
    &\quad+ \lambda_m - \mu_m \\ 
    &= |\lambda| - |\mu| + \sum_{i=1}^{m-1} (m-i) (\lambda_i - \lambda^{(m-1)}_i)
\end{aligned}    
\]
We recognize that \(\lambda_i - \lambda^{(m-1)}_i = n(\tau)_{z_i-z_m}\) and re-sum, setting \(n_{(a,b)} \equiv n(\tau)_{z_a-z_b}\) for convenience.
\[
    \sum_{i=1}^{m-1} (m-i) (\lambda_i - \lambda_i^{(m-1)}) \,=\,\,
    \begin{aligned}
    &n_{(1,m)} +  \\
    &n_{(1,m)}+ n_{(2,m)} + \\
    &n_{(1,m)}+ n_{(2,m)}+ n_{(3,m)}+ \\
    &\,\,\,\vdots \\
    &n_{(1,m)}+ n_{(2,m)}+ n_{(3,m)} + \cdots + n_{(m-1,m)} 
    \end{aligned} 
\]
Observe that, since \(n_{(a,b)}\) is just the number of \(m\) in row \(i\), \(n_{(1,m)} + \cdots + n_{(i,m)} = \mu_m -\) the number of \(m\) in the last \(m-i+1\) rows. Summing the latter terms counts the number of 
\(m\) in row \(i\) exactly \(i\) times, for \(i = 1,\ldots, m\). 
But
\[
    \sum_{k=1}^{\mu_m} r_{m,k} = \sum_{i=1}^{\lambda_1}i\cdot(\textrm{the number of boxes numbered }m\textrm{ in row }i)
\] 
too. Thus upon adding \(0 = \mu_m -\mu_m\) to the re-summation we get
\(m\mu_m - \sum_{k=1}^{\mu_m} r_{m,k}\)
as expected.

Similarly, the fibres of the maps \(X_{\tau^{(i+1)}}\to X_{\tau^{(i)}}\) for \(i = 1,\ldots,m-2\) have dimension 
\[
(i+1)\mu_{i+1} - \sum_{k = 1}^{\mu_{i+1}} r_{i+1,k} = 
\langle \lambda^{(i+1)} - \mu^{(i+1)},\rho^{(i+1)}\rangle - \langle \lambda^{(i)} - \mu^{(i)},\rho^{(i)}\rangle \,.
\]
Since these are the differences making up the telescoping sum
\[
\begin{aligned}
    \langle \lambda - \mu,\rho\rangle 
    &= \langle \lambda - \mu,\rho\rangle  - \langle \lambda^{(m-1)} - \mu^{(m-1)} ,\rho^{(m-1)}\rangle \\
    & + \langle \lambda^{(m-1)} - \mu^{(m-1)} ,\rho^{(m-1)}\rangle - \langle \lambda^{(m-2)} - \mu^{(m-2)} ,\rho^{(m-2)}\rangle + \cdots 
\end{aligned}
\] 
it follows that \(\dim X_\tau = \langle \lambda - \mu,\rho\rangle \).
\end{proof}
\begin{conjecture}
The map in \Cref{irrfib} is a trivial fibration. Consequently \(X_\tau^d = X_\tau\) and \(Z_\tau = \overline{X_\tau} \) is defined by a recurrence \(Z_\tau \cong Z_{\tau-\young(m)} \times \CC^{m-r}\) for \(r\) equal to the row coordinate of the last \(m\) in \(\tau\). 
\end{conjecture}
\subsection{Decomposing \texorpdfstring{\(\ot\cap\n\)}{O cap T cap n}}
\label{ss:decom}
We conclude the first part of this paper with a proof of \Cref{mainanne1}.
\begin{theorem}
\label{anne1}
The map 
\(\tau\mapsto Z_\tau\) is a bijection of 
\(\cT(\lambda)_\mu\) and irreducible components of \(\ot \cap \n\). Moreover, \(\dim Z_\tau = \dim X_\tau = \langle \lambda - \mu,\rho\rangle\).
\end{theorem}
\begin{proof}
    Let \(A\in \ot\cap\n\) be generic for a component and consider the tableau \(\tau\) obtained from the GT-pattern of Jordan types of submatrices \(A\big|_{V^{(i)}}\) for \(1\le i\le m\). Then \(\tau\in\cT(\lambda)_\mu\) and \(A\in Z_\tau\).
\end{proof}
\section{Equations of \mvi cycles}
\label{s:equat}
Let \(G(\cK)\to \Gr : g\mapsto g G(\cO)\) denote the quotient map, and when there is no confusion, set \([g] \equiv g G(\cO)\).
Let \(\Gr_\mu = G_1\xt[t^{-1}]L_\mu\) for \(G_1\xt[t^{-1}] = \Ker(G\xt[t^{-1}]\xrightarrow{t = \infty}G)\).
\subsection{The \mvy isomorphism}
\label{ss:mviso}
\begin{theorem}
\label{mvy}
\cite{cautis2018categorical}
The map \(\phi: \TT_\mu\cap\cN\to G_1\xt[t^{-1}]t^\mu\) defined by 
\[
\begin{aligned}
    \phi(A) &= t^\mu + a(t) \\
    a_{ij}(t) &= - \sum A_{ij}^k t^{k-1} \\
    A_{ij}^k &= k\textrm{th entry from the left of the }\mu_j\times\mu_i\textrm{ block of }A
\end{aligned}
\]
yields the \mvy isomorphism of type \((\lambda,\mu)\)
\(\Psi : \ot\to \overline{\Gr^\lambda}\cap\Gr_\mu\) 
defined by \(\Psi = [\phi(A)]\).
\end{theorem}
\begin{proof}
The reader can check that, given \([g]\in \Gr\), the map
\[
[g]\mapsto
\left[t\big|_{\cO^m/g^T\cO^m}\right]_B
\]
for \( 
    B = ([e_m],\ldots, [e_m t^{\mu_m -1}],\ldots , [e_1], \ldots,[e_1 t^{\mu_1 -1}])\) is a two-sided inverse.
\end{proof}
\begin{example}
    \label{mvyimage}
Let \(\tau = \young(112,23)\) as before. Then
\[
    \phi(Z_\tau) = \left\{
        g =
\begin{pmatrix}
    t^2 & 0 & 0 \\
    -a-bt & t^2 & 0 \\
    -c & -d & t
\end{pmatrix} : a,d = 0
    \right\}\,.
\]
\end{example}
\begin{corollary}
\label{resmvy}
The restriction \(\psi = \Psi\big|_{\ot\cap \n}\) is an isomorphism of \(\ot\cap \n\) and \(\overline{\Gr^\lambda}\cap S^\mu_-\) which we'll refer to as the restricted \mvy isomorphism of type \((\lambda,\mu)\).
\end{corollary}
\begin{proof}
Suppose \(A\in \TT_\mu \cap\n\). Then \(\phi(A) \in G_1\xt[t^{-1}]t^\mu\cap N_-(\cK)t^\mu = N_-(\cK)t^\mu\) so \(\psi(A) \in S^\mu_- \). Conversely, if \(\phi(A)\in N_-(\cK)t^\mu\), then \(A\in \TT_\mu\cap \n\). Since \(\dim \ot\cap\n = \dim \overline{\Gr^\lambda}\cap S^\mu_-\) and \(\Psi\) is onto, we can conclude that \(\Im \psi = \overline{\Gr^\lambda}\cap S^\mu_-\).
\end{proof}
\subsection{Equal Lusztig data}
\label{ss:equal}
Let \(V\) be a finite dimensional complex vector space. In \cite{kamnitzer2010mirkovic} the author works in the right quotient \(\Gr^T = G(\cO)\backslash G(\cK) \) where he defines the Lusztig datum of an MV cycle using the valuation
\[
\val: V\otimes\cK\to \ZZ: v\mapsto k \text{ if }v\in V\otimes t^k \cO \setminus V\otimes t^{k+1} \cO\,.
\]
Let \(w\in W\). Let \(\omega_i \in X_\bullet(T)\) denote the \(i\)th fundamental weight, \(\omega_i = z_1 + \cdots + z_i\) for \(1\le i\le m\). Fix a highest weight vector \(v_{\omega_i}\) in the \(i\)th fundamental irreducible representation \(L(\omega_i)\) of \(G\) and consider
\[
    D_{w\omega_i}^T : \Gr^T\to\ZZ : G(\cO)g \mapsto \val (g\widetilde{w} \cdot v_{\omega_i})
\] 
with \(\widetilde w\) denoting the lift of \(w\) to \(G\).
These functions cut out the semi-infinite cells \(\flip{S}^\nu_w = G(\cO)t^\nu U_w(\cK) \subset \Gr^T\) for \(U_w = \widetilde w U \widetilde{w}^{-1}\) as follows.
\begin{lemma}[\cite{kamnitzer2010mirkovic} Lemma 2.4]
\label{kamlem}
\(D_{w\omega_i}^T\) is constructible and
\[
    \flip{S}^\nu_w = \{L\in\Gr^T : D_{w\omega_i}^T(L) = \langle \nu, w\omega_i\rangle \text{ for all } i\}
\]
\end{lemma}
By considering the transpose map \(\Gr\to \Gr^T: gG(\cO)\mapsto G(\cO)g^T\) we derive an analogous result for \(S^\nu_w=U_w(\cK) t^\nu G(\cO) \subset\Gr \).
\begin{lemma}
\label{easylemma}
If \(gG(\cO)\in S^\nu_w\subset \Gr\), then \(G(\cO)g^T\in \flip{S}^{\nu}_{ww_0}\subset \Gr^T\). 
In particular, the order on the vertices of the MV polytope of the MV cycle to which \(gG(\cO)\) belongs is reversed, with the datum \(\nu_\bullet = (\nu_w)_{w\in W}\) for \(gG(\cO)\), defining \(\nu_\bullet^T = (\nu_{ww_0})_{w\in W}\) for \(G(\cO)g^T\) and \(S^\nu_w = \{L\in\Gr : D^T_{ww_0\omega_i}(L^T) = \langle\nu,ww_0\omega_i\rangle\}\).
\end{lemma}
\begin{proof}
Let \(gG(\cO) \in S^\nu_w\subset\Gr\). Then \(g = wnw^{-1}t^\nu\) for some \(n\in U\). Then
\[
\begin{aligned}
    g^T &= t^\nu (w^{-1})^T n^T w^T \\
        &= t^\nu w (w_0 n w_0) w^{-1} \in t^\nu U_{ww_0}(\cK)
\end{aligned}    
\]
The computation
\[
\begin{aligned}
    g^T ww_0\cdot v_{\omega_i} &= t^\nu ww_0 nw_0 w^{-1} ww_0 \cdot v_{\omega_i} \\
                         &= t^\nu ww_0\cdot v_{\omega_i} = t^{\langle \nu, ww_0\omega_i\rangle}\cdot  ww_0 v_{\omega_i} 
\end{aligned}    
\]
checks that 
\(D^T_{ww_0\omega_i} (G(\cO)g^T) = \val(g^T w w_0 \cdot v_{\omega_i})\) 
is equal to 
\(\langle \nu,ww_0\omega_i\rangle\) 
agreeing with \Cref{kamlem}.
\end{proof}

We define \(D_{w\omega_i} \) on \(\Gr\) by 
\begin{equation}
    D_{w\omega_i}(gG(\cO)) = D^T_{ww_0\omega_i}(G(\cO)g^T)
\end{equation}
and rewrite
\[
S^\nu_w = \{L\in \Gr : D_{w\omega_i}(L) = \langle\nu,ww_0\omega_i\rangle \text{ for all } i\}\,.
\]

The Lusztig datum \(n_\bullet\) of an MV cycle in \(\Gr^T\) is defined by
\begin{equation}
    \label{geold}
    n(\phantom{L})_{z_a - z_b} = D_{[a\cdots b]}^T(\phantom{L}) - D_{[a+1\cdots b]}^T(\phantom{L}) - D_{[a\cdots b-1]}^T(\phantom{L}) + D_{[a+1 \cdots b-1]}^T(\phantom{L})
\end{equation}
on generic elements.
Here \([a{\cdots} b]\) is shorthand for the permutation of \(\omega_{b-a+1}\) that has 1s in positions \(a\) through \(b\) and zeros elsewhere. If \(b<a\) we understand \(D_{[ab]}^T \equiv 0\). 

Given \(1\le b\le m \), let \(F_b = \Sp_{\cO}(e_1,\ldots,e_b)\subset\cK^m\), let \(w_0^b\) be the longest element in the Weyl group of \(G_b\), and let \(\omega_i^b\) be the \(i\)th fundamental weight in \(\ZZ^b\), the weight lattice of \(G_b\). 
Let \(g\in U_{-}(\cK)\). Then \(g^T \in U(\cK)\) and 
\[
    D^T_{[ab]}(G(\cO)g^T) = 
    D^T_{[ab]}(G_b(\cO)(g^T\big|_{F_b})) =
    D^T_{w_0^b\omega^b_{b-a+1}}(G_b(\cO)(g^T\big|_{F_b})) \,.
\]
In turn
\[
    D^T_{w_0^b\omega^b_{b-a+1}}(G_b(\cO)(g^T\big|_{F_b})) 
    = D_{\omega^b_{b-a+1}}((g^T\big|_{F_b})^TG_b(\cO)) \,.
\]
In particular, if \((g^T\big|_{F_b})^TG_b(\cO) \in S^\nu\cap S^\eta_-
\subset G_b(\cK)/G_b(\cO)\), then \Cref{easylemma} implies that
\[
    D_{w\omega^b_{b-a+1}}((g^T\big|_{F_b})^TG_b(\cO)) = \begin{cases}
        \langle \eta, \omega^b_{b-a+1}\rangle & w = e \\
        \langle \nu, w\omega^b_{b-a+1}\rangle & w = w_0^b
    \end{cases}\,.
\]

Now fix \(\tau\in\cT(\lambda)_\mu\) and \(d = \langle\lambda-\mu,\rho\rangle\), and consider the inclusion \(\iota: \overline{\Gr^\lambda}\cap S^\mu_- \hookrightarrow \overline{\Gr^\lambda}\cap \overline{S^\mu_-}\). Note that \( \overline{\Gr^\lambda}\cap \overline{S^\mu_-} = \overline{\Gr^\lambda\cap S^\mu_-}\cup X\) where every component of \(X\) has dimension strictly less than \(d\). (This follows from decompositions of \(\overline{\Gr^\lambda}\) and \(\overline{S^\mu_\pm}\) afforded by \cite{mirkovic2007geometric}).
\begin{lemma}
\label{whereisg} For a dense subset of \(A\in Z_\tau\),
\(\psi(A) \in S^\lambda\cap S^\mu_{-}\).
\end{lemma}
\begin{proof}
By \Cref{resmvy}, \(\psi(Z_\tau)\) is dense in a component of \(\overline{\Gr^\lambda}\cap S^\mu_-\) which has pure dimension \(d\) by \cite[Theorem 3.2]{mirkovic2007geometric}. Thus \(Z = \overline{\iota(\overline{\psi(Z_\tau)})}\) is an MV cycle. By \cite[\S 3.3]{kamnitzer2010mirkovic}, \(Z\cap S^\lambda\) is dense in \(Z\), so \(\psi^{-1}(\iota^{-1}(Z\cap S^\lambda))\) is dense in \(Z_\tau\).
\end{proof}
For the purpose of the next lemma, let \(\psi_b: \overline{\OO}_{\lambda^{(b)}}\cap\TT_{\mu^{(b)}}\cap\n_b\to \overline{\Gr^{\lambda^{(b)}}}\cap S^{\mu^{(b)}}_-\) denote the restricted \mvy isomorphism of type \((\lambda^{(b)},\mu^{(b)})\)
with \(\n_b \subset \mat(N_b,\CC)\) denoting the subalgebra of \(N_b\times N_b\) upper-triangular matrices for \(N_b = \mu_1 + \cdots + \mu_b\) and let \(\pi_{b} :\ot\cap\n\to \overline{\OO}_{\lambda^{(b)}}\cap\TT_{\mu^{(b)}}\cap\n_b \) denote the restriction \(\pi_{b}(A) = A\big|_{V^{(b)}}\) with \(1\le b\le m\). 
\begin{lemma}
\label{dgammatau}
For \(1\le a < b\le m\), the generic value of \(D_{w_0\omega_{b-a+1}}\circ\iota\circ\psi\) on \(Z_\tau\) is 
\(
    \langle\lambda^{(b)},w_0^b \omega^b_{b-a+1}\rangle
\).
\end{lemma}
\begin{proof}
Let \(1\le a < b\le m\). We apply \Cref{whereisg} to \(\psi_{b}\) and let \(W_b\) be the dense subset of \(B\in Z_{\tau^{(b)}}\) for which 
\(\psi_b(B)\in S^{\lambda^{(b)}} \cap S^{\mu^{(b)}}_-\). Then \(\pi^{-1}_{b}(W_b) \) is dense in \(Z_\tau\).
Moreover, for \(A\in\pi_b^{-1}(W_b)\), \(\phi(A)^T\big|_{F_b} = \phi_b(\pi_{b}(A))^T\), so by \Cref{easylemma},
\(
    G_b(\cO)\phi_b(\pi_{b}(A))^T\in  S^{\lambda^{(b)}}_- \cap S^{\mu^{(b)}}
\)
and
\(
    D_{w_0\omega_{b-a+1}}(\iota(\psi(A))) =
        \langle \lambda^{(b)},w_0^b\omega^b_{b-a+1}\rangle
\).
\end{proof}

We are just about ready to verify \Cref{mainanne2}. Following \cite{BERENSTEIN1988453}, define the Lusztig datum \(n_\bullet\) of \(\tau\in \cT(\lambda)_\mu\) by the zig-zag differences
\begin{equation}
    \label{comld}
    n(\tau)_{z_a-z_b} = \lambda^{(b)}_a - \lambda^{(b-1)}_a 
\end{equation}
for \(1\le a<b\le m\).
\begin{theorem}
\label{anne2} 
\(Z = \overline{\iota(\overline{\psi(Z_\tau)})}\) is an MV cycle of coweight \((\lambda,\mu)\) having Lusztig datum \(n(\tau)_\bullet\) given by \Cref{comld} above.
\end{theorem}
\begin{proof}
Let \(1\le a < b\le m\). By \Cref{dgammatau}, the Lusztig datum of \(Z\), as defined by \Cref{geold}, is given by
\[
\begin{aligned}
    n(Z)_{z_a - z_b} &= 
    \langle\lambda^{(b)}, w_0^b\omega_{b-a+1}\rangle - \langle\lambda^{(b)},w_0^b\omega_{b-a}\rangle \\
    &\quad - \langle\lambda^{(b-1)},w_0^{b-1}\omega_{b-a}\rangle + \langle\lambda^{(b-1)},w_0^{b-1}\omega_{b-a-1}\rangle \\
    &= \langle\lambda^{(b)},w_0^b z_{b-a+1}\rangle - \langle\lambda^{(b-1)},w_0^{b-1}z_{b-a}\rangle \\
    &= \lambda^{(b)}_a - \lambda^{(b-1)}_a\,.
\end{aligned}
\]
\end{proof}
%
%
%
\bibliographystyle{alpha}

\hfill

\end{document}